\documentclass{amsart}
\usepackage{mathabx,mathtools,amsrefs,mathrsfs,math}
\usepackage{tikz}
\usetikzlibrary{arrows,automata,calc}

\renewcommand\tree{{\mathscr T}}

\begin{document}
\title{Self-similar products of groups}
\date{May 12th, 2018}

\author{Laurent Bartholdi}
\address{L.B.: D\'epartement de Math\'ematiques et Applications, \'Ecole Normale Sup\'erieure, Paris \textit{and} Mathematisches Institut, Georg-August Universit\"at zu G\"ottingen}
\email{laurent.bartholdi@gmail.com}

\author{Said N. Sidki}
\address{S.S.: Departamento de Matematica, Universidade de Bras\'ilia, Bras\'ilia DF 70910-900, Brazil}
\email{ssidki@gmail.com}

\thanks{This work is supported by the ``@raction'' grant ANR-14-ACHN-0018-01. The second author thanks the Mathematics Department of the \'Ecole Normale Sup\'erieure for the warm hospitality during April of 2018}

\begin{abstract}
  We address the problem of determining the class of self-similar
  groups, and in particular its closure under restricted direct
  products. We show that the group $\Z^{(\omega)}$ is self-similar,
  that $G^{(\omega)}\rtimes C_2$ is self-similar whenever $G$ is, and
  that permutational wreath products of a finite abelian group with a
  self-similar group are self-similar.
\end{abstract}
\maketitle

\section{Introduction}
A self-similar group is a group $G$ admitting a faithful, self-similar
action on a regular\footnote{This condition will be relaxed
  in~\S\ref{ss:ssg}.} rooted tree $\tree$. By this we mean that $G$
acts transitively on the neighbours of the root, that the tree
$\tree_v$ hanging from any such neighbour $v$ is isomorphic to
$\tree$, and that the action of the stabilizer $G_v$ on $\tree_v$,
when conjugated back to $\tree$ via the isomorphism of trees,
coincides with the original action.

Very little is known on which groups admit such an action, and for now
only sporadic constructions exist --- more often than not, a group is
self-similar if it is already given as such. Notable exceptions are
general constructions of linear or affine groups oven residually
finite rings~\cites{brunner-s:glnz,nekrashevych-s:sc}, and nilpotent
groups admitting
dilations~\cite{berlatto-sidki:nilpotent,nekrashevych-pete:scale-invariant,bondarenko-kravchenko:nilpotent}. In
particular, all $2$-step finitely generated torsion-free nilpotent
groups are self-similar~\cite{berlatto-sidki:nilpotent}; but lattices
in Dyer's example~\cite{dyer:nilpotent} of a Lie group without
dilations cannot be self-similar.  Free groups
are self-similar~\cite{vorobets:free}, but $\Z\wr\Z$ is
not~\cite{dantas-sidki:wreath}*{Theorem~1}.

On the one hand, a self-similar group must be residually finite; but
few other necessary conditions are known. In particular, a
self-similar group need not have solvable word
problem~\cite{bartholdi:wordorder}.

Is is also not yet well understood under which operations the class of
self-similar groups is closed. We note in~\S\ref{ss:barakah} that if
$G$ is self-similar, then $G^d$ is self-similar for all $d\in\N$. Let
$G^{(\omega)}$ denote the restricted direct product of countably many
copies of $G$. We attract attention to the following question:
\begin{question}\label{question:main}
  Let $G$ be a self-similar group. Is $G^{(\omega)}$ self-similar?
\end{question}

In~\S\ref{ss:barakah}, we show in Proposition~\ref{prop:c2} that if
$G$ is self-similar, then $G^{(\omega)}\rtimes C_2$ is also
self-similar, for any free action of $C_2$ on $\omega$.  We then
answer Question~\ref{question:main} positively for $G=\Z$:
\begin{thm}\label{thm:zomega}
  There exists a self-similar action of $\Z^{(\omega)}$ on the binary
  rooted tree. However, there does not exist any finite-state such
  action, nor any such action with a non-trivial semi-invariant
  subgroup (see~\ref{ss:second} for definitions).
\end{thm}

This extends readily to $G$ any finitely generated abelian
group. However, we do not know if $\Z_2^{(\omega)}$ is self-similar,
nor $\Z^\omega$, the unrestricted direct product.

Finally, in~\S\ref{ss:lagniappe} we give a fairly general construction
of self-similar wreath products with a finite abelian group, see
Proposition~\ref{prop:wreath}.

There are at least four essentially equivalent languages with which
one may describe self-similar groups: ``bisets'', ``wreath
recursions'', tree actions, and ``virtual endomorphisms'' as in this
text. One of our aims is to show how powerful and concise this
language can be.

\section{Self-similar groups}\label{ss:ssg}
Consider a self-similar group $G$ acting on the right on a spherically
homogeneous rooted tree $\tree$. Assume that the action of $G$ is
transitive on the neighbours of the root $\lambda$ of $\tree$, and
consider a vertex $v$ at distance $1$ from $\lambda$. The
self-similarity means, in this notation, that the subtree $\tree_v$
growing at $v$ embeds (say by an isometry $\tau\colon\tree_v\to\tree$)
into the whole tree, and we have $G_v\tau\subseteq \tau G$ as maps
$\tree_v\to\tree$ composed left-to-right. Abstracting away from the
tree, the group $G$ admits a finite-index subgroup $H=G_v$ and a
homomorphism $f\colon H\to G$ given by $\tau\circ h=h^f\circ\tau$ qua
maps $\tree_v\to\tree$. We summarize the data as
$f\colon G\ge H\to G$, and call $f$ the \emph{virtual endomorphism} of
the action.

Starting from a virtual endomorphism $f\colon G\ge H\to G$, a tree
action may readily be constructed. Set $H_0=G$ and
$H_n=H_{n-1}^{f^{-1}}$ for all $n\ge1$. We define $\tree$ as the coset
tree of the $H_n$'s: $\tree=\bigsqcup_{n\ge0}H_n\backslash G$, with
natural right action of $G$ by translation. Note that if $f$ is
surjective, then we have $[G:H]=[H_{n-1}:H_n]\coloneqq m$ for all $n$,
so $\tree$ is an $m$-regular rooted tree. The vertex $H=H_1$ is at
distance $1$ from the root $H_0$, the subtree $\tree_H$ is
$\tree\cap H$, and the isomorphism $\tree_H\to\tree$ is given by
$H_n g\mapsto H_{n-1}g^f$ for all $H_n g\subseteq H$. Set also
$H_\omega=\bigcap_{n\ge0}H_n$; then $H_\omega$ is the stabilizer in
$\partial\tree$ of the ray $(H_0,H_1,\dots)$. We call it the
\emph{parabolic subgroup} of the action. We note that $\ker(f)$ is
contained in $H_\omega$.
\begin{lem}\label{lem:inj}
  The induced map
  $f\colon H\backslash H_\omega\to G\backslash H_\omega$ is injective.
\end{lem}
\begin{proof}
  If $a,b\in H$ are such that $H_\omega a^f=H_\omega b^f$, then
  $b a^{-1}\in H_\omega$ by maximality of $H_\omega$, so
  $H_\omega a=H_\omega b$.
\end{proof}

The assumption that the action be faithful translates to the map $f$
being \emph{core-free}, in the sense that there is no non-trivial
subgroup $K\le H$ with $K\triangleleft G$ and $K^f\le K$. Indeed,
there is evidently a maximal such $K$, called the \emph{$f$-core} of
$H$, which coincides with the kernel of the action of $G$ on
$\tree$. If the series $H_n$ is normal, then $H_\omega$ coincides with
the $f$-core of $H$.

Consider a virtual endomorphism $f\colon G\ge H\to G$, and choose a
right transversal $T$ to $H$, so $G=H T$; for $g\in G$, write
$\overline g\in T$ its representative in $T$. The map $f$ extends to a
homomorphism $\hat f\colon G\to G^T\rtimes\sym T$ by the rule
\[a^{\hat f} = (t\mapsto (t a(\overline{t a})^{-1})^f)\cdot (t\mapsto \overline{t a});
\]
this is essentially the contents of the Kaloujnine-Krasner
theorem~\cite{kaloujnine-krasner:extensions}. An element $g$ is called
\emph{finite-state} there is a finite subset $S\subseteq G$ containing
$g$ such that $S^{\hat f}\subseteq S^T\times\sym T$. A self-similar
group is finite-state if all its elements are finite-state; it is
enough to check this on a generating set. Note that a self-similar
group may be finite-state for one choice of transversal and not
finite-state for another choice.

A basic example of self-similar group is the \emph{adding machine}:
the group $G=\langle a\rangle\cong\Z$ with $H=\langle a^2\rangle$ and
$(a^{2n})^f=a^n$. The tree consists of all arithmetic progressions
with stride a power of $2$, connected by inclusion, and $\Z$ acts
naturally by translation. This action is finite-state: for the
transversal $T=\{0,1\}$ and $g=a^n$, one may choose
$S=\{1,a,\dots,a^n\}$. The same argument shows that $\Z_p$ is
self-similar for every prime $p$.

\section{Barakah}\label{ss:barakah}

For $G_1,G_2$ two self-similar groups, their direct product
$G_1\times G_2$ is self-similar: given $f_i\colon G_i\ge H_i\to G_i$
core-free, one has
$(f_1,f_2)\colon G_1\times G_2\ge H_1\times H_2\to G_1\times G_2$
core-free, and
$[G_1\times G_2:H_1\times H_2]=[G_1:H_1]\cdot[G_2:H_2]$. This shows
that $G^d$ is self-similar for all $d<\omega$.

The construction for $G^d$ may be made more ``economical'' by
considering $\dot f\colon G^d\ge H\times G^{d-1}\to G^d$ given by
$(h_1,g_2,\dots,g_d)^{\dot f}=(g_2,\dots,g_d,h_1^f)$.

As a partial step in the direction of showing that $G^{(\omega)}$ is
self-similar as soon as $G$ is, we give the following general result,
inspired by a construction of Dantas and
Sidki~\cite{dantas-sidki:wreath}*{Theorem~2}:
\begin{prop}\label{prop:c2}
  Let $G$ be a self-similar group. Then $G^{(\omega)}\rtimes C_2$ is
  self-similar, for any free action of $C_2$ on $\omega$.
\end{prop}
\begin{proof}
  Let the self-similarity of $G$ be expressed by a core-free map
  $f\colon G\ge H\to G$. Fixing notation, consider
  $\dot G=G^{(\N)}\rtimes\langle\sigma|\sigma^2\rangle$ with
  $(g_0,g_1,\dots,g_{2n},g_{2n+1},\dots)^\sigma=(g_1,g_0,\dots,g_{2n+1},g_{2n},\dots)$.

  Define $\dot H=H\times G^{(\N-\{0\})}$, so
  $[\dot G:\dot H]=2[G:H]$, and $\dot f\colon\dot H\to\dot G$ by
  \[(a_0,a_1,a_2,a_3,\dots)^{\dot f}=(a_0^f,a_2,a_1,a_4,a_3,\dots).\]

  Consider a normal subgroup $K\triangleleft\dot G$ with $K\le \dot H$
  and $K^{\dot f}\le K$; we will show $K=1$. Assume on the contrary
  that $K$ contains a non-trivial element
  $\dot a=(a_0,\dots,a_n,\dots)$. Since
  $\dot a^{\dot f}=(a_0^f,\dots)$ and by assumption $f$ is core-free,
  we have $a_0=1$, namely $K\le1\times G^{(\N-\{0\})}$. Let
  now $n\in\N$ be minimal such that $a_n\neq1$. If $n$ is odd, then
  $\dot a^\sigma=(1,\dots,1,a_n,\dots)$ with the `$a_n$' in position
  $n-1$; and $\dot a^\sigma\in K$ because $K$ is normal in $G$. If
  $n>0$ is even, then $\dot a^{\dot f}=(1,\dots,1,a_n,\dots)$ with
  again the `$a_n$' in position $n-1$; and $\dot a^{\dot f}\in K$
  because $K$ is $f$-invariant. In all cases, we obtain an element of
  $K$ with smaller $n$, and eventually an element $\dot a\in K$ with
  $a_0\neq 1$, a contradiction.
\end{proof}

\section{Proof of the first part of Theorem~\ref{thm:zomega}, existence of an action}

We fix once and for all a model of the binary rooted tree $\tree$: its
vertex set is $\{0,1\}^*$, and there is an edge from $v$ to $v x$ for
all $v\in\{0,1\}^*$ and all $x\in\{0,1\}$. The root is the empty
sequence $\lambda$, there is a vertex $0$ at distance $1$ from
$\lambda$, and the map $\tau\colon\tree\to\tree_0$ is given by
$v\mapsto 0v$.

The multiplicative group of $\Z_2$ is $1+2\Z_2$. Choose any
$\eta\in2\Z_2^\times$; namely a $2$-adic that is
$\equiv2\pmod4$. Every $a\in\Z_2$ then admits a unique
\emph{base-$\eta$ representation}:
\[a=\sum_{i\ge0} a_i \eta^i,\qquad a_i\in\{0,1\}.\]

The boundary of $\tree$ is naturally identified with $\{0,1\}^\omega$,
and also with $\Z_2$ under the map
\begin{equation}\label{eq:boundary}
  x_0 x_1\dots\leftrightarrow\sum_{i\ge0} x_i \eta^i.
\end{equation}

We have a natural action of $\Z_2$ on the boundary $\partial\tree$,
given by translation; and we claim that this action comes from a
self-similar action of $\Z_2$ on $\tree$. For this, we write the
groups $G=\Z_2$ and $H=2\Z_2$ and $f\colon H\to G$ given by
$a\mapsto a/\eta$. It is straightforward to see that the action of
$\Z_2$ on the boundary coincides with the action of $\Z_2$ by
translation on itself, via the identification~\eqref{eq:boundary}.

We consider now the additive subgroup $G\coloneqq\Z[1/\eta]\cap\Z_2$
of $\Z_2$, and claim that it is a self-similar subgroup of $\Z_2$.
The subgroup $H$ here is $G\cap2\Z_2$, with $[G:H]=2$, and the map $f$
is the restriction to $H$ of the original map $f\colon 2\Z_2\to\Z_2$
given by $a\mapsto a/\eta$. We have $H^f\subseteq G$, since
\[H^f=(\Z[1/\eta]\cap2\Z_2)^f=(\Z[1/\eta]\cap2\Z_2)/\eta\subseteq\Z[1/\eta]\cap\Z_2=G.\]
The Theorem's first claim follows since the action of $G$ on $\tree$
is faithful, and $G\cong\Z^{(\omega)}$ as soon as $\eta$ is
transcendental.

Note that the construction of $f$ may be made quite explicit, as
follows. We may start with $G=\Z^{(\N)}$ and
$H=2\Z\times\Z^{(\N-\{0\})}$, and define the virtual endomorphism
$f\colon H\to G$ by
$(2a_0,a_1,a_2,\dots)\mapsto(a_0+\alpha_1 a_1+\alpha_2
a_2+\cdots,a_0,a_1,a_2,\dots)$ for appropriate choices of
$\alpha_n\in\{0,1\}$, defined as follows. The $0$th basis vector of
$G$ is addition by $1$, and for $n\ge1$ the $n$th basis vector of $G$
is addition by $p_n(1/\eta)$ for some integral polynomial $p_n$ of
degree at most $n$ chosen such that $p_n(1/\eta)\in 2\Z_2$. One starts
with $p_0(t)=2$ and chooses $\alpha_n\in\{0,1\}$ such that
$p_{n+1}(t)\coloneqq t p_n(t)-\alpha_{n+1}$ satisfies
$p_{n+1}(1/\eta)\in 2\Z_2$.

\section{Proof of the second part of Theorem~\ref{thm:zomega}, state-closed subgroups}\label{ss:second}

We prove a more general statement than that claimed in
Theorem~\ref{thm:zomega}. Given a self-similar group
$f\colon G\ge H\to G$, a \emph{semi-invariant} subgroup is $G_0\le G$
with $(G_0\cap H)^f\le G_0$. The self-similar action of $G$ induces a
self-similar action $f\restriction\colon G_0\ge G_0\cap H\to G_0$ on
$G_0$ and, if $G_0$ is normal, a quotient action
$\overline f\colon G/G_0\ge H/(H\cap G_0)\to G/G_0$. The following
result may be of independent interest:
\begin{prop}\label{prop:semiinvariant}
  Let $f\colon G\ge H\to G$ define a self-similar action of the group
  $G\cong\Z^{(\omega)}$. If there exists a finite-rank semi-invariant
  subgroup $G_0\le G$ with $[G:H]=[G_0:G_0\cap H]$, then $f$ has a
  non-trivial core.
\end{prop}
\begin{proof}
  Without loss of generality, we may assume that $G_0$ has minimal
  rank among the semi-invariant subgroups with
  $[G:H]=[G_0:G_0\cap H]$.  Set $H_0\coloneqq H\cap G_0$ and
  $f_0\coloneqq f\restriction H_0\colon H_0\to G_0$.  Denote $G_0$'s
  isolator subgroup by
  $\sqrt{G_0}=\{a\in G\mid a^n\in G_0\text{ for some }n\neq0\}$. Then
  $\sqrt{G_0}$ has the same rank as $G_0$, it has a complement in $G$,
  and is also semi-invariant; so we may replace $G_0$ by its
  isolator. Let $G_1$ denote a complement of $G_0$, so we have
  $G=G_0\times G_1$ and $H=H_0\times G_1$ and may write in block
  matrix form
  \[f=\begin{pmatrix}f_0 & 0\\ g & f_1\end{pmatrix}\text{ with }f_1\colon G_1\to G_1\text{ and }g\colon G_1\to G_0.\]

  Since $G_0$ is finite-dimensional, $f_0$ admits a minimal polynomial
  $\chi\in\Z[t]$. By minimality of $G_0$'s rank, no factor of $\chi$
  is monic. If $\chi(f)=0$, then $f_1$ would be block triangular with
  diagonal blocks factors of $f_0$; now these diagonal blocks have
  monic minimal polynomial, a contradiction. Therefore $\chi(f)\neq0$,
  so there exists $v\in [H:G] G_1$ with $v\chi(f)\neq0$. Set
  \[K\coloneqq \langle v\chi(f)f^n\colon n\ge0\rangle.\]
  We claim that $K$ is an $f$-invariant normal subgroup of $H$.

  To aid the computations, let us introduce a ``differential
  operator'' $d\colon\Z[t]\to\Z\langle f_0,g,f_1\rangle$ given by
  \[d(1)=0,\quad d(t)=g,\quad d(p\cdot q)=dp\cdot q(f_0)+p(f_1)\cdot
    dq.
  \]
  We then have $f_1 d\chi = d(t\chi)=d(\chi t)=\chi(f_1)g+d\chi f_0$,
  and deduce
  \[\chi(f)=\begin{pmatrix}0&0\\d\chi & \chi(f_1)\end{pmatrix},\text{ and more generally }\chi(f)f^n=\begin{pmatrix}0&0\\ f_1^n d\chi & f_1^n\chi(f_1)\end{pmatrix};\]
  the first computation follows directly from matrix multiplication,
  and the second one by applying $n$ times the identity
  $f_1 d\chi=\chi(f_1)g+d\chi f_0$.
  
  Now $v f_1^n\in [G:H]G_1$ for all $n$, so
  $v f_1^n d\chi\in [G:H]G_0\le H_0$ for all $n$, and therefore
  $K\le H$. The invariance $K^f\le K$ is automatic, and normality
  holds trivially because $G$ is abelian.
\end{proof}

It immediately follows from Proposition~\ref{prop:semiinvariant} that
no element of $G$ may be finite-state: its stateset would generate a
semi-invariant, finite-rank subgroup of $G$. The proof of
Theorem~\ref{thm:zomega} is finished.

It would be interesting to know if the same result holds in the
following generality: $G\cong A^{(\omega)}$ for a self-similar group
$A$, and $G_0$ a subgroup of $A^n\le A^{(\omega)}$ for some
$n<\omega$.

\section{Lagniappe}\label{ss:lagniappe}

Let us finally give in what appears to be a more natural setting the
main result of~\cite{dantas-sidki:wreath}. Recall that for $H\le G$ we
denote by $H\backslash G$ the set of left cosets, namely the set of
$H g$ with $g\in G$.
\begin{prop}\label{prop:wreath}
  Let $G$ be self-similar with parabolic subgroup $H_\omega$, and let
  $A$ be a finite abelian group. Then
  $A^{(H_\omega\backslash G)}\rtimes G$ is self-similar, and is
  finite-state whenever $G$ is.
\end{prop}
Often $A^{(H_\omega\backslash G)}\rtimes G$ is called a ``restricted
permutational wreath product'', or a ``lamplighter group''; the
terminology comes from the picture of a road network identified with
$H_\omega\backslash G$, with lamps at each position; an element of the
group is a sequence of movements of the lamplighter (controlled by
$G$) and changes to the intensity of the lamp in front of him/her
(controlled by $A$).
\begin{proof}
  Let the self-similarity of $G$ be expressed as $f\colon G\ge H\to G$.
  Set $\dot G\coloneqq A^{(H_\omega\backslash G)}\rtimes G$ and
  \[\dot H\coloneqq\Big\{\phi\colon H_\omega\backslash G\to A\text{ finitely supported },\prod_{H_\omega a\in H_\omega\backslash G}\phi(H_\omega a)=1\Big\}\rtimes H.
  \]
  Note that we have $\dot H=[A^{(H_\omega\backslash G)},G]\rtimes H$
  and $[\dot G:\dot H]=\#A\cdot[G:H]$. We define $\dot f$ as follows:
  \[(\phi,h)^{\dot f}=(H_\omega a\mapsto\prod_{(H_\omega b)^f\subseteq
      H_\omega a}\phi(H_\omega b),h^f).
  \]
  Note that the product contains at most one term, by
  Lemma~\ref{lem:inj}.

  Consider a normal subgroup $K\triangleleft\dot G$ with $K\le \dot H$
  and $K^{\dot f}\le K$; we will show $K=1$. Firstly, the assumption
  that $f$ be core-free implies $K\le A^{(H_\omega\backslash G)}$. Assume
  therefore that $K$ contains a non-trivial element
  $\phi\colon H_\omega\backslash G\to A$ with $\prod\phi=1$. In particular,
  the support of $\phi$ contains at least two points. Choose $\phi$
  with support of minimal cardinality.

  Since $K$ is normal in $\dot G$, we may conjugate $\phi$ by $G$ so
  as to ensure $\phi(H_\omega)\neq1$; we let $S$ denote its support,
  written as a union of $H_\omega$-cosets in $G$, so we have
  $H_\omega\subseteq S$. Write $\phi_0\coloneqq\phi$, and consider
  $\phi_1\coloneqq\phi_0^{\dot f}$; it is given by
  $\phi_1(H_\omega a^f)=\phi_0(H_\omega a)$ for all $a\in H$, extended
  by the identity away from $H_\omega H^f$. In particular,
  $\phi_1(H_\omega)\neq1$ so $\phi_1\neq1$, and the support of
  $\phi_1$ is $(S\cap H)^f$. Continuing in this manner with
  $\phi_n=\phi_{n-1}^{\dot f}$, the support of $\phi_n$ is
  $H_\omega(S\cap H_n)^f$, and always contains $H_\omega$. Since the support
  of $\phi_n$ is by assumption of same cardinality as $S$, we have
  $S\subseteq H_n$ for all $n$, so $S\subseteq H_\omega$, and
  therefore $S=H_\omega$, contradicting the fact that the support of
  $\phi$ contains at least two cosets of $H_\omega$.

  Assume finally that $G$ is finite-state for the choice of
  transversal $T\subseteq G$; let $t_0\in T$ be the representative of
  $H$. We write
  $\dot A\coloneqq\{\phi\colon H_\omega\backslash G\to A\text{
    supported at }H_\omega\}\cong A$, write $\dot a$ for the element
  of $\dot A$ with value $a$ at $H_\omega$, and choose
  $\dot A\times T$ as transversal for $\dot H$ in $\dot G$. Since
  $\dot G$ is generated by $\dot A\cup G$, it suffices to show that
  these elements are finite-state. For this, we compute $\hat{\dot f}$
  on them.  If $g\in G$ and $g^{\hat f}=(t\mapsto g_t)\cdot\pi$ with
  $\pi\in\sym T$, then
  \[g^{\hat{\dot f}}=\big((\dot a,t)\mapsto g_t\big)\cdot\big((\dot a,t)\mapsto t^\pi\big)\]
  while if $b\in A$ then
  \[(\dot b)^{\hat{\dot f}}=\left((\dot a,t)\mapsto\Big\{\begin{array}{l}\dot a\text{ if }t=t_0\\1\text{ else}\end{array}\Big\}\right)\cdot\big((\dot
    a,t)\mapsto\dot{ab}\big).
  \]
  In the first case, the set of states of $g$ is the same in $G$ and
  in $\dot G$, while in the second case the set of states of $\dot a$
  is $\{\dot a,1\}$.
\end{proof}

There are free, self-similar actions of $\Z^d$ for all $d<\omega$, so
in particular $\Z^d$ can be made self-similar with trivial parabolic
subgroup. This recovers~\cite{dantas-sidki:g_p_d}*{Theorem~4} stating
that $C_p\wr\Z^d$ is self-similar. Their finite-state example for
$d=p=2$ is essentially the same as the one given by the above
proposition.

\end{document}